\documentclass[10pt]{article}

\usepackage{mathptmx,amsmath,amscd,amssymb,amsthm,xspace}
\usepackage[all,tips]{xy}
\usepackage[dvips]{graphicx}
\usepackage{graphics,epsfig,wrapfig,verbatim,syntonly}
\usepackage{hyperref,amssymb,color, url,fancyhdr,tikz-cd}

\theoremstyle{plain}
\newtheorem{thm}{Theorem}[section]
\newtheorem{cor}[thm]{Corollary}
\newtheorem{prop}[thm]{Proposition}
\newtheorem{lemma}[thm]{Lemma}

\theoremstyle{definition}

\theoremstyle{remark}
\newtheorem{rem}{Remark}

 \newtheorem{rmk}[subsection]{Remark}

\DeclareMathOperator{\PSL}{PSL}
\DeclareMathOperator{\SL}{SL}\DeclareMathOperator{\Aut}{Aut}

\DeclareMathOperator{\Comm}{Comm}
\DeclareMathOperator{\Isom}{Isom}
\DeclareMathOperator{\SO}{SO}
\DeclareMathOperator{\SU}{SU}

\DeclareMathOperator{\PGL}{PGL}

\DeclareMathOperator{\Res}{Res}
\DeclareMathOperator{\Ind}{Ind}\DeclareMathOperator{\Hom}{Hom}
\DeclareMathOperator{\End}{End}\DeclareMathOperator{\Gal}{Gal}
\DeclareMathOperator{\Var}{Var}
\DeclareMathOperator{\Cor}{Cor}\DeclareMathOperator{\Mot}{Mot}
\DeclareMathOperator{\CH}{CH}
\DeclareMathOperator{\SB}{SB}\DeclareMathOperator{\Bl}{Bl}
\DeclareMathOperator{\Grp}{Grp}\DeclareMathOperator{\CoInd}{CoInd}

\newcommand{\bdef}{\overset{\text{def}}{=}}

\newcommand{\LL}{\mathbb{L}}

\newcommand{\innp}[1]{\left< #1 \right>}
\newcommand{\abs}[1]{\left\vert#1\right\vert}
\newcommand{\set}[1]{\left\{#1\right\}}

\newcommand{\Hy}{\mathbf{H}}
\newcommand{\N}{\mathbf{N}}
\newcommand{\Q}{\mathbf{Q}}
\newcommand{\R}{\mathbf{R}}
\newcommand{\Z}{\mathbf{Z}}
\newcommand{\C}{\mathbf{C}}
\newcommand {\spec} {{\mathrm{Spec}\,}}
\newcommand{\F}{\mathbf{F}}
 \newcommand {\PP} {{\mathbf P}}

\begin{document}
\bibliographystyle{plain}

%
\title{\textbf{Integral Gassman equivalence of \\ algebraic and hyperbolic manifolds}}
\author{D.~Arapura\thanks{Purdue University, West Lafayette, IN. E-mail: \tt{arapura@purdue.edu}},~~J.~Katz\thanks{Purdue University, West Lafayette, IN. E-mail: \tt{katz9@purdue.edu}},~~ D. B. McReynolds\thanks{Purdue University, West Lafayette, IN. E-mail: \tt{dmcreyno@purdue.edu}},~~P.~Solapurkar\thanks{Purdue University, West Lafayette, IN. E-mail: \tt{psolapur@purdue.edu}}  }
\maketitle

\begin{abstract}
\noindent In this paper we construct arbitrarily large families of smooth projective varieties and closed Riemannian manifolds that share many algebraic and analytic invariants. For instance, every non-arithmetic, closed hyperbolic $3$--manifold admits arbitrarily large collections of non-isometric finite covers which are strongly isospectral, length isospectral, and have isomorphic integral cohomology where the isomorphisms commute with restriction and co-restriction. We can also construct arbitrarily large collections of pairwise non-isomorphic smooth projective surfaces where these isomorphisms in cohomology are natural with respect to Hodge structure or as Galois modules. In particular, the projective varieties have isomorphic Picard and Albanese varieties, and they also have isomorphic effective Chow motives. Our construction employs an integral refinement of the Gassman--Sunada construction that has recently been utilized by D. Prasad. One application of our work shows the non-injectivity of the map from the Grothendieck group of varieties over $\overline{\Q}$ to the Grothendieck group of the category of effective Chow motives. We also answer a question of D. Prasad. 
\end{abstract}

\section{Introduction}

The main purpose of the present article is the construction of geometric objects which share a large class of algebraic, analytic, geometric, and topological invariants. Our main tool is a refinement of a construct that dates back to Gassman which has been utilized by Perlis \cite{Perlis}, Sunada \cite{sunada}, and others. Given a commutative ring $R$ with identity, a  group $G$, and a pair of subgroups $P_1,P_2 \leq G$, we say that $P_1,P_2$ are \textbf{$R$--equivalent} if $R[G/P_1]$ and $R[G/P_2]$ are isomorphic as left $R[G]$--modules. When $G$ is finite and $P_1,P_2$ are $\Q$--equivalent, the associated triple $(G,P_1,P_2)$ is called a \textbf{Gassman triple}. For general $R$, $(G,P_1,P_2)$ is called an \textbf{$R$--Gassman triple}. 

Scott \cite{scott} found non-conjugate $\Z$--equivalent subgroups $P_1,P_2$ of $\PSL(2,\mathbf{F}_{29})$. The subgroups $P_1,P_2$ are isomorphic to $\mathrm{Alt}(5)$ and are conjugate in $\PGL(2,\mathbf{F}_{29})$. D. Prasad \cite{prasad} recently employed this $\Z$--Gassman triple $(\PSL(2,\mathbf{F}_{29}),P_1,P_2)$ to construct non--isometric Riemann surfaces with isomorphic Jacobian varieties viewed only as unpolarized abelian varieties. He also constructed a pair of non--isomorphic finite extensions of $\Q$ with isomorphic idele class groups and adele rings. In particular, these finite extensions are arithmetically  equivalent (i.e.~have the same Dedekind zeta functions). Recently, the third author with B.~Linowitz and N.~Miller \cite{LMM} used non--isomorphic fields with isomorphic adele rings to construction isomorphisms between various Galois cohomology sets that arise in the study of $K$--forms of semisimple Lie groups. One instance of this was the construction of an isomorphism between the Brauer groups of the fields which was compatible with the restriction and co-restriction maps. The bijections between other Galois cohomology sets was also compatible with respect to the restriction and co-restriction maps. 


\subsection{Differential geometric examples}

Our results split across algebraic and differential geometry. We state our differential geometric results first. Before doing so, we require some additional notation and terminology. Given a closed, Riemannian manifold $M$ with associated Laplace--Beltrami operator $\triangle_M$, the operator $\triangle_M$ acts on the space of $L^2$ functions or $L^2$ $k$--forms $\Omega^k(M)$. We denote the associated eigenvalue spectrum for the operator $\triangle_M$ acting on $\Omega^k(M)$ by $\mathcal{E}_k(M)$. In the case of $k=0$, we denote the eigenvalue spectrum by $\mathcal{E}(M)$ and refer to this as the \textbf{eigenvalue spectrum}. The spectrum $\mathcal{E}(M)$ is a well studied analytic invariant of the Riemannian manifold $M$ and is known to determine the dimension, volume, and total scalar curvature. A related geometric  invariant is the \textbf{primitive geodesic length spectrum} $\mathcal{L}_p(M)$ of $M$. Assuming for simplicity that $M$ is negatively curved, each free homotopy class of closed curves on $M$ has a unique geodesic representative. We define $\mathcal{L}_p(M)$ to be the set of lengths (with multiplicity) of each geodesic representative in each free homotopy class. We denote by $H^k(M,\Z)$ the \textbf{$k$th singular cohomology group} of $M$ with trivial $\Z$--coefficients. Given a finite cover $M' \to M$, we have induced homomorphisms $\Res\colon H^k(M,\Z) \to H^k(M',\Z)$ and $\Cor\colon H^k(M',\Z) \to H^k(M,\Z)$. For a pair of finite covers $M_1,M_2 \to M$, we say that a morphism $\psi_k\colon H^k(M_1,\Z) \to H^k(M_2,\Z)$ is \textbf{compatible} if the diagram
\begin{equation}\label{Eq:Natural}
\begin{tikzcd}
& H^k(M,\Z) \arrow[rd,"\Res", bend left = 30] \arrow[ld,"\Res"', bend right = 30] & \\ H^k(M_1,\Z) \arrow[rr,"\psi_k"', bend right = 45] \arrow[ru,"\Cor"', bend right = 30]  & & H^k(M_2,\Z) \arrow[lu,"\Cor", bend left = 30]&
\end{tikzcd}
\end{equation}
commutes. Finally, $M$ is called \textbf{large} if there exists a finite index subgroup $\Gamma_0 \leq \pi_1(M)$ and a surjective homomorphism of $\Gamma_0$ to a non-abelian free group. We now state our first result and refer the reader to \S 2 for a brief review of real/complex hyperbolic manifolds and the definition of non-arithmetic manifolds.

\begin{thm}\label{Sec:Intro:T1}
Let $M$ be a closed hyperbolic $n$--manifold that is large and non-arithmetic. Then for each $j \in \N$, there exist pairwise non-isometric, finite Riemannian covers $M_1,\dots,M_j$ of $M$ such that the following holds:
\begin{itemize}
\item[(1)]
$\mathcal{E}_k(M_i) = \mathcal{E}_k(M_{i'})$ for all $k$ and all $i,i'$.
\item[(2)]
$\mathcal{L}_p(M_i) = \mathcal{L}_p(M_{i'})$ for all $i,i'$.
\item[(3)]
There exist compatible isomorphisms $\psi_k\colon H^k(M_i,\Z) \to H^k(M_{i'}, \Z)$ for all $k$ and for all $i,i'$.
\end{itemize}
\end{thm} 

When $n \geq 3$, it follows from Mostow--Prasad rigidity (see \cite{Mostow}, \cite{GPra}) that $\pi_1(M_i),\pi_1(M_{i'})$ are non-isomorphic for $i \ne i'$. When manifolds $M_i,M_{i'}$ satisfy (1), they are referred to as \textbf{strongly isospectral}. When only $\mathcal{E}(M_i) = \mathcal{E}(M_{i'})$, the pair is said to be \textbf{isospectral}. Similarly, when (2) holds, the pair is said to be \textbf{length isospectral}. We note that for every $n \geq 2$, by work of Gromov--Piatetski-Shapiro \cite{GPS}, there are infinitely many commensurability classes of examples for which Theorem \ref{Sec:Intro:T1} can be applied. Moreover, being non-arithmetic or large are both commensurability invariants. 

\begin{rem}
The compatible isomorphism in singular cohomology with trivial $\Z$--coefficients is a special case of a more general result that relates the cohomology of manifolds $M_1,M_2$ that arise from this refined Gassman/Sunada construction; see Lemma \ref{Sec:Coho:L4} which also answers Question 2 in \cite{prasad}. In particular, there is a large class of coefficients for which compatible isomorphisms exist and the coefficients need not be trivial. 
\end{rem}

That one can construct non-isometric manifolds that satisfy (1) and (2) has been known since \cite{sunada}; see also \cite{LMNR} for a variant on \cite{sunada}. Additionally, it was known that when two manifolds arise from this construction, besides satisfying (1) and (2), they have $H^k(M_i,\Q) \cong H^k(M_{i'},\Q)$. However, it need not be the case that the pair have isomorphic integral cohomology. Bartel--Page \cite{BartelPage1} (see also \cite{BartelPage}) found examples of pairs arising from a Sunada construction which do not have isomorphic cohomology groups with coefficients in $\mathbf{F}_p$. Specifically, given a finite set of primes $S$, there exist strongly isospectral closed hyperbolic $3$--manifolds with non-isomorphic $\mathbf{F}_p$--cohomology for every $p \in S$ and isomorphic $\mathbf{F}_p$--cohomology for every $p \notin S$ (see \cite[Thm 1.2]{BartelPage1}). Also, Lauret--Miatello--Rossetti \cite{LMR} prove that strongly isospectral pairs need not have isomorphic cohomology rings.

By work of Agol \cite[Thm 9.2]{Agol}, every closed hyperbolic $3$--manifold is large; that hyperbolic surfaces are large is well known. As a result, we obtain the following corollary of Theorem \ref{Sec:Intro:T1}.

\begin{cor}\label{Sec:Intro:C1}
Let $M$ be a closed, non-arithmetic real hyperbolic $2$-- or $3$--manifold. Then for each $j \in \N$, there exist pairwise non-isometric, finite Riemannian covers $M_1,\dots,M_j$ of $M$ such that the following holds:
\begin{itemize}
\item[(1)]
$\mathcal{E}_k(M_i) = \mathcal{E}_k(M_{i'})$ for all $k$ and all $i,i'$.
\item[(2)]
$\mathcal{L}_p(M_i) = \mathcal{L}_p(M_{i'})$ for all $i,i'$.
\item[(3)]
There exist compatible isomorphisms $\psi_k\colon H^k(M_i,\Z) \to H^k(M_{i'}, \Z)$ for all $k$ and for all $i,i'$.
\end{itemize}
\end{cor} 

When $n=2$, Corollary \ref{Sec:Intro:C1} is a small generalization of \cite{prasad}. In this setting, Borel \cite{borel} proved that there are only finite many non-isometric arithmetic hyperbolic surfaces of area at most $A$ for any $A>0$. In particular, for each genus $g \geq 2$, there are only finitely many points in $\mathcal{M}_g$, the moduli space of hyperbolic structures on $\Sigma_g$, that correspond to arithmetic hyperbolic structures. However, $\mathcal{M}_g$ has real dimension $6g-6$ and so we see that a typical hyperbolic structure on $\Sigma_g$ is non-arithmetic. A closed hyperbolic $3$--manifold is typically non-arithmetic as well. For each positive real number $V >0$, Borel \cite{borel} proved that there are only finitely many non-isometric arithmetic hyperbolic $3$--manifolds of volume at most $V$. However, it follows by work of Thurston that when $V$ is sufficiently large, there exist infinitely many closed hyperbolic $3$--manifolds of volume at most $V$. For instance, if $M_0$ is the complement of the figure-eight knot, for all but finitely many Dehn surgeries on $\partial M_0$, the resulting closed $3$--manifold will admit a complete hyperbolic structure by Thurston's Dehn Surgery theorem (see \cite{Thurston}). The figure-eight knot complement also admits a complete hyperbolic structure on its interior and the volumes of the closed hyperbolic manifolds obtained by Dehn surgery on $M_0$ are strictly smaller than $\mathrm{Vol}(M_0)$. Consequently, only finitely many of these closed hyperbolic $3$--manifolds can be arithmetic by Borel's finiteness theorem. For $n \geq 4$, the number of non-isometric, complete, finite volume hyperbolic $n$--manifolds of volume at most $V$ is finite by Wang \cite{Wang}. In this case, we can count the number of non-isometric complete, finite volume hyperbolic $n$--manifolds of volume at most $V$. Restricting to only the arithmetic or non-arithmetic manifolds, we obtain two counting functions and it is known that these functions have the same growth type (see \cite{GL} and the references therein for more on this topic). 

Returning to the main topic of this subsection, we end with another family of examples.

\begin{cor}\label{Sec:Intro:C2}
Let $M$ be a closed complex hyperbolic $2$--manifold that is non-arithmetic and large. Then for each $j \in \N$, there exist pairwise non-isometric, finite Riemannian covers $M_1,\dots,M_j$ of $M$ such that:
\begin{itemize}
\item[(1)]
$\mathcal{E}_k(M_i) = \mathcal{E}_k(M_{i'})$ for all $k$ and all $i,i'$.
\item[(2)]
$\mathcal{L}_p(M_i) = \mathcal{L}_p(M_{i'})$ for all $i,i'$.
\item[(3)]
There exist compatible isomorphisms $\psi_k\colon H^k(M_i,\Z) \to H^k(M_{i'},\Z)$ for all $k$ and for all $i,i'$.
\end{itemize}
\end{cor} 

By work of Deligne--Mostow \cite{DM}, there are commensurability classes of complex hyperbolic 2--manifolds for which Corollary \ref{Sec:Intro:C2} can be applied. At present, there are only finitely many known commensurability classes of non-arithmetic complex hyperbolic 2--manifolds; see \cite{DPP} for more on this topic. 

\subsection{Algebro-geometric results and examples}\label{Sec:AGresults}

We now describe some results that relate various algebro-geometric invariants for pairs of smooth projective varieties that are constructed via $R$--equivalence for certain rings $R$. The examples from Corollary \ref{Sec:Intro:C2} provide non-trivial examples of such pairs for $R=\Z$. Large families of examples of non-isomorphic smooth projective varieties for $R=\Q$ were constructed in \cite[Thm 1.1]{McR}. These examples arise in all possible dimensions and the universal cover of these examples can be taken to be any irreducible, non-compact Hermitian symmetric space. 

For a field $K \subseteq \C$, we denote the category of smooth projective varieties over $K$ by $\Var_K$. The set of complex points $X(\C)$ of a variety $X\in \Var_K$ can be regarded as a complex manifold. These spaces carry the usual topological invariants such as the (topological) fundamental group or singular cohomology. However, the singular cohomology of an algebraic variety $X\in \Var_K$ is endowed with more structure than just an abelian group. Hodge theory provides the singular cohomology groups with a canonical decomposition $H^i(X,\Z)\otimes \C= \bigoplus_{p+q=i}H^{pq}$ such that $\overline{H^{pq}} = H^{qp}$ (see \cite{voisin} for instance). Such a decomposition is referred to as a \textbf{Hodge structure}. The subspace $H^{pq}$ can be defined as the space of de Rham cohomology classes represented by closed complex valued differential forms of type $(p,q)$. If $X$ is equipped with a K\"ahler metric, then $H^{pq}$ is isomorphic to the space of harmonic $(p,q)$--forms. For $k$ odd, the Hodge structure can be used to construct a complex structure on the real torus $H^k(X(\C),\R)/H^k(X(\C),\Z)$ which turns it into a complex torus called the \textbf{Griffiths intermediate Jacobian}. When $k=1, 2\dim_\C(X)-1$, these tori are in fact abelian varieties called the \textbf{Picard} and \textbf{Albanese varieties} of $X$. Setting $\Q_p$ to be the field of $p$--adic numbers and $\Z_p$ to be the ring of $p$--adic integers, via the comparison isomorphisms with \'etale cohomology (see \cite{milne}), we have
\begin{align*}
H^k(X(\C),\Z_p) &\cong H_{et}^k(X_{\overline{K}},\Z_p) :=\varprojlim_n H_{et}^k(X_{\overline{K}},\Z/(p^n))  , \\
H^k(X(\C),\Q_p) &\cong H_{et}^k(X_{\overline{K}},\Q_p) := H_{et}^k(X_{\overline{K}},\Z_p)\otimes \Q_p.
\end{align*}
where $\overline{K}$ denotes the algebraic closure of $K$ and $X_{\overline{K}}= X\times_{\spec K} \spec{\overline{K}}$. The \'etale cohomology groups carry natural $\Gal(\overline{K}/K)$--actions which encode important arithmetic information about $X$. When $K$ is a number field, these Galois modules determine the Hasse--Weil zeta function (see \cite{serre}). At a more basic level, the fundamental homological invariant of a variety is its motive; see \S \ref{Sec:motives} for more details. Rational cohomology with its Hodge  structure or $\Q_p$--\'etale cohomology with its Galois action depend only on  the motive. 

\begin{thm}\label{thm:cohXi}
Suppose $(G,P_1,P_2)$ is an $R$--Gassman triple, $p\colon X\to Y$ is a Galois \'etale cover with $X,Y \in \Var_K$ and with Galois group $G$, and $X_i= X/P_i$ for $i=1,2$.
\begin{enumerate}
\item[(1)]   
If $K=\C$, then there is an $R$--module isomorphism of singular cohomology groups $H^i(X_1(\C),R)\cong H^i(X_2(\C),R)$. If $R=\Z$ (resp.~$\Q$), then the isomorphism respects the canonical integral (resp.~rational) Hodge structures. In particular, the intermediate Jacobians of $X_i$ are isomorphic (resp.~isogenous).

\item [(2)]
If $R=\Z_p$ (resp. $\Q_p$) and $\overline{K}$ is the algebraic closure of $K$, then there is a $\Gal(\overline{K}/K)$--equivariant isomorphism of \'etale cohomology $H_{et}^*(X_{1,\overline{K}},\Z_p)\cong H_{et}^*(X_{2,\overline{K}},\Z_p)$ (resp. $H_{et}^*(X_{1,\overline{K}},\Q_p)\cong H_{et}^*(X_{2,\overline{K}},\Q_p)$). 

\item[(3)]
If $R=\Q$, then the effective Chow motives $M(X_i)$ of $X_i$  are isomorphic.
\end{enumerate}
\end{thm}

\begin{rem}
The last statement of case (1), when $\dim X_i=1$, is due to Prasad \cite{prasad}. In case (2) and $R=\Q_p$, this result is due to Prasad--Rajan \cite{pr}, who also observed that this implies that the Hasse--Weil zeta functions agree when $K$ is a number field. Note that case (3) actually implies the previous two statements when $\Q \subseteq R$.
\end{rem}

Combining this theorem with  (the proof of) Corollary \ref{Sec:Intro:C2} yields:

\begin{thm}\label{prop:ZequivVar}
Fix an embedding $\overline{\Q} \subset \C$. Then for every $j \in \N$, there exists smooth projective surfaces $X_1,\ldots, X_j$ defined over $\overline{\Q}$ such that
 
\begin{enumerate}
\item[(1)]  $H^k(X_i(\C),\Z)\cong H^k(X_{i'}(\C),\Z)$ as Hodge structures for all $k$ and all $i,i'$. 
\item[(2)] $H_{et}^k(X_i,\Z_p)\cong H_{et}^k(X_{i'},\Z_p)$ as Galois modules for all $k$ and all $i,i'$. 
\item[(3)] The motives $M(X_i)\cong  M(X_{i'})$ all $i,i'$. 
\item[(4)] The topological fundamental groups of $X_i(\C)$ are pairwise non-isomorphic.
\end{enumerate}
\end{thm}

In \S \ref{Sec:motives}, we will use Theorem \ref{prop:ZequivVar} to show the non-injectivity of the map from the Grothendieck group of varieties over $\overline{\Q}$ to the Grothendieck group of the category of effective Chow motives (see Theorem \ref{Sec:Motive:NonInject}). The differences $[X_i]-[X_{i'}]$ will give nonzero elements in the kernel.

\paragraph{\textbf{Acknowledgments.}} The authors would like to thank Nick Miller, Deepam Patel, Alan Reid, and Matthew Stover for conversations on the topics in this paper. We would also like to thank the referee for helpful comments that helped improve the clarify of the article. The first author was partially supported by an NSF grant. The third author was partially supported by NSF grant DMS-1408458.

\section{Preliminaries}

Real and complex hyperbolic $n$--space are examples of symmetric spaces of non-compact type. We refer the reader to \cite{Goldman} and \cite{Ratcliffe} for a thorough introduction to these spaces. The isometry group $\Isom(\Hy_\R^n)$ of real hyperbolic $n$--space $\Hy_\R^n$ is isogenous to the subgroup $\SO(n,1)$ of $\SL(n+1,\R)$ that preserves the bilinear form
\[ B_{n,1}(x,y) = -x_{n+1}y_{n+1} + \sum_{j=1}^n x_jy_j. \]
Given a discrete subgroup $\Gamma \leq \Isom(\Hy_\R^n)$, the quotient space $\Hy_\R^n/\Gamma$ is a real hyperbolic $n$--orbifold. When $\Gamma$ is torsion free (i.e. contains no non-trivial elements of finite order), the quotient space is a complete, real hyperbolic $n$--manifold. We say that $\Gamma$ is a \textbf{lattice} if $\Hy_\R^n/\Gamma$ has finite volume. If $\Hy_\R^n/\Gamma$ is also compact, we say that $\Gamma$ is \textbf{cocompact}. Conversely, given a complete, finite volume real hyperbolic $n$--manifold $M$, via the action of $\pi_1(M)$ on the universal cover $\Hy_\R^n$, we obtain an injective homomorphism $\pi_1(M) \to \Isom(\Hy_\R^n)$. The image under this representation is a lattice. We note that because this representation depends on the choice of a lift $\widetilde{p} \in \Hy_\R^n$ of the base point $p \in M$, this representation is unique only up to conjugation in $\Isom(\Hy_\R^n)$. 

The isometry group $\Isom(\Hy_\C^n)$ of complex hyperbolic $n$--space $\Hy_\C^n$ is isogenous to the subgroup $\SU(n,1)$ of $\SL(n+1,\C)$ that preserves the hermitian form
\[ H_{n,1}(w,z) = -w_{n+1}\overline{z_{n+1}} + \sum_{j=1}^n w_j\overline{z_j}. \]
Complex hyperbolic $n$--manifolds and orbifolds are constructed similarly to those in the real hyperbolic setting but taking discrete subgroups of $\Isom(\Hy_\C^n)$. One important difference between real and complex hyperbolic $n$--manifolds that will be relevant is the existence of complex projective structures. First, a complex hyperbolic $n$--manifold is a complex manifold of real dimension $2n$. Due to an exceptional isogeny between $\SL(2,\R)$ and $\SU(1,1)$, real hyperbolic $2$--manifolds coincide with complex hyperbolic $1$--manifolds. In particular, real hyperbolic $2$--manifolds come with a natural complex structure. For all $n>2$, real hyperbolic $n$--manifolds are not naturally complex. When $\Gamma$ is a torsion free cocompact lattice in $\SU(n,1)$, the associated complex hyperbolic $n$--manifold is a non-singular, complex projective algebraic variety. 

Taking $\mathrm{G}$ to be either $\Isom(\Hy_\R^n)$ or $\Isom(\Hy_\C^n)$, given a pair of subgroups $\Gamma_1,\Gamma_2 \leq \mathrm{G}$, we say that $\Gamma_1,\Gamma_2$ are \textbf{commensurable} if $\Gamma_1 \cap \Gamma_2$ is a finite index subgroup of both $\Gamma_1,\Gamma_2$. We define the commensurator of $\Gamma$ in $\mathrm{G}$ to be the subgroup
\[ \Comm(\Gamma) = \set{g \in \mathrm{G}~:~g^{-1}\Gamma g, \Gamma \textrm{ are commensurable}}. \]
One sees that $\Gamma \leq \Comm(\Gamma)$. It follows from work of Margulis \cite[Thm 1, pp.~2]{Margulis} that either $[\Comm(\Gamma):\Gamma] < \infty$ or $\Comm(\Gamma)$ is dense in $\mathrm{G}$ in the analytic topology. When $[\Comm(\Gamma):\Gamma]$ is finite, we say that $\Gamma$ is \textbf{non-arithmetic} and when $\Comm(\Gamma) \leq \mathrm{G}$ is dense, we say that $\Gamma$ is \textbf{arithmetic}. 

We end this section with a short remark concerning the cohomology/homology of $\Gamma$ and its associated real or complex hyperbolic manifold. When $\Gamma$ is discrete and torsion free, the associated manifold $M = \Hy_\R^n/\Gamma$ or $\Hy_\C^n/\Gamma$ is a $K(\Gamma,1)$--space for $\Gamma$ since $\Hy_\R^n,\Hy_\C^n$ are contractible. As a result, we can establish the cohomology isomorphisms for the spaces by establishing them for the cohomology of the associated lattices. 

\section{Isomorphisms in group cohomology}

In this section, we record some basic results that relate the group cohomology of $\Z$--equivalent subgroups of finite and infinite groups. We refer the reader to \cite{Brown} for a more complete treatment of group cohomology.

Given a group $G$ and a subgroup $P \leq G$, we denote the restriction functor by $\Res^G_P$. Restriction has left and right adjoints given by the induction and co-induction functors $\Ind_P^G,\CoInd_P^G$. Explicitly, for a $\Z[P]$--module $A$, the underlying modules are $\Ind_P^G(A) =\Z[G] \otimes_{\Z[P]} A$ and $\CoInd_P^G(A) =\Hom_{\Z[P]}(\Z[G],A)$, with respective $G$ actions given by the $\Z$--linear extensions of $g\cdot (x\otimes a)=gx\otimes a$ and $(g\cdot \phi)(x)= \phi(xg)$. 

We start with a pair of well known results.
	
\begin{lemma}\label{Sec:Coho:L1}
If $G$ is a group and $P \leq G$ is finite index, then induction and co-induction are isomorphic as $\Z[G]$--modules. 
\end{lemma}

\begin{lemma}\label{Sec:Coho:L2}
If $G$ is a group, $P \leq G$ is of finite index, and $A$ is a $\Z[G]$--module, then $\CoInd_P^G(\Res_P^G(A))=A\otimes_{\Z} \Z[G/P]$.
\end{lemma}

We note that $P_{1},P_{2} \leq G$ are $\Z$--equivalent if and only if $\CoInd_{P_1}^G(\Res_{P_1}^G(1))$, $\CoInd_{P_2}^G(\Res_{P_2}^G(1))$ are isomorphic as $\Z[G]$--modules. Given a $\Z[G]$--module $A$, we say that a morphism $\psi_k\colon H^k(P_1,\Res_{P_1}^G(A)) \to H^k(P_2,\Res_{P_2}^G(A))$ is \textbf{compatible} if the diagram

\begin{equation}\label{Eq:Natural2}
\begin{tikzcd}
& H^k(G,A) \arrow[rd,"\Res", bend left = 30] \arrow[ld,"\Res"', bend right = 30] & \\ H^k(P_1,\Res_{P_1}^G(A)) \arrow[rr,"\psi_k"', bend right = 45]  \arrow[ru, "\CoInd"', bend right = 30]& & H^k(P_2,\Res_{P_2}^G(A)) \arrow[lu,"\CoInd", bend left = 30]
\end{tikzcd}
\end{equation}
commutes.

\begin{lemma} \label{Sec:Coho:L3}
Let $G$ be a finite group and $P_{1}, P_{2} \leq G$ be $\Z$--equivalent subgroups. Then for any $\Z[G]$--module $A$ and any nonnegative integer $k$, there is a compatible isomorphism $H^{k}(P_{1},\Res_{P_1}^G(A))\to H^{k}(P_{2},\Res_{P_2}^G(A))$.
\end{lemma}

\begin{proof}
By Shapiro's lemma (see \cite[III.8]{Brown}), we have $H^{k}(P_{i},\Res_{P_i}^G(A))=H^{k}(G,\CoInd_{P_i}^G(\Res_{P_i}^G(A)))$. By Lemma \ref{Sec:Coho:L2}, the coefficients for the latter cohomology groups are $A \otimes_{\Z} \Z[G/P_i]$, viewed as $\Z[G]$--modules. Since $P_{1}$ and $P_{2}$ are $\Z$--equivalent, these coefficient modules are $\Z[G]$--isomorphic. Thus, the right hand side of the equality above is actually independent of $i$, providing the isomorphism as claimed. Compatibility follows from the naturality of the  isomorphism in Shapiro's lemma. Specifically, upon choosing an isomorphism of the $\Z[G]$--modules $\Z[G/P_1]$ and $\Z[G/P_2]$, isomorphisms in cohomology groups
\[ H^k(P_1,\Res_{P_1}^G(A)) \to H^{k}(G,\CoInd_{P_1}^G(\Res_{P_1}^G(A))) \to H^{k}(G,\CoInd_{P_2}^G(\Res_{P_2}^G(A))) \to H^k(P_2,\Res_{P_1}^G(A)) \]
are induced by isomorphisms of coefficients.  
\end{proof}

We now deduce a few corollaries of the above. First, we observe that if $\psi\colon \Gamma \to G$ is a surjective homomorphism and $\Gamma_i = \psi^{-1}(P_i)$ for $\Z$--equivalent subgroups $P_1,P_2 \leq G$, then $\Gamma_1,\Gamma_2 \leq \Gamma$ are also $\Z$--equivalent subgroups. In particular, via the previous subsection, we obtain the following lemma.
\begin{lemma}\label{Sec:Coho:L4}
Let $\psi\colon \Gamma \to G$ be a surjective homomorphism, $P_1,P_2 \leq G$ be $\Z$--equivalent subgroups, and $\Gamma_i = \psi^{-1}(P_i)$. Then for any $\Z[\Gamma]$--module $A$ and any nonnegative integer $k$, there is a compatible isomorphism $H^k(\Gamma_1,\Res_{\Gamma_1}^\Gamma(A)) \to H^k(\Gamma_2,\Res_{\Gamma_2}^\Gamma(A))$.
\end{lemma}

One case of Lemma \ref{Sec:Coho:L4} of particular interest is when $A$ is a trivial $\Z[\Gamma]$--module (i.e. the $\Gamma$--action is trivial).

\begin{cor}\label{Sec:Coho:C1}
Let $\psi\colon \Gamma \to G$ be a surjective homomorphism, $P_1,P_2 \leq G$ be $\Z$--equivalent subgroups, and $\Gamma_i = \psi^{-1}(P_i)$. Then for any trivial $\Z[\Gamma]$--module $A$ and any nonnegative integer $k$, there is a compatible isomorphism $H^k(\Gamma_1,A) \to H^k(\Gamma_2,A)$.
\end{cor}

We note that one deficiency of Lemma \ref{Sec:Coho:L4} is the requirement that our initial module $A$ be a $\Z[\Gamma]$--module. This prevents us from obtaining a bijection between the $\Z[\Gamma_1]$--modules and $\Z[\Gamma_2]$--modules in a way that induces compatible isomorphisms in group cohomology. 

\section{Proof of Theorem \ref{Sec:Intro:T1}}

In this section, we prove Theorem \ref{Sec:Intro:T1}, Corollary \ref{Sec:Intro:C1}, and Corollary \ref{Sec:Intro:C2}. 

\subsection{Algebraic construction}

Throughout this section, for each $r \in \N$, we will denote the free group of rank $r$ by $F_r$. The main goal of this section is the following construction of arbitrarily large families of finite index subgroups of certain lattices that are pairwise non-isomorphic and pairwise $\Z$--equivalent. 

\begin{prop}\label{TechConProp}
Let $\mathrm{G}$ be a simple Lie group that is not isogenous to $\SL(2,\R)$ and let $\Gamma \leq \mathrm{G}$ be a lattice that is large and non-arithmetic. Then for each $j\in \N$, there exist finite index subgroups $\Delta_1,\dots,\Delta_j \leq \Gamma$ such that
\begin{itemize}
\item[(a)]
The subgroups $\Delta_i$ are pairwise non-isomorphic.
\item[(b)]
The subgroups $\Delta_i$ are pairwise $\Z$--equivalent.
\end{itemize}
\end{prop}

We note that Proposition \ref{TechConProp} holds when $\mathrm{G}$ is isogenous to $\SL(2,\R)$ but with (a) changed to the following:
\begin{itemize}
\item[(a')]
The subgroups $\Delta_i$ are pairwise non-conjugate in $\mathrm{G}$.
\end{itemize} 

This subsection is devoted to the proof Proposition \ref{TechConProp}. We start with a basic lemma on the size of the set $\mathrm{Hom}_{\mathrm{sur}}(F_r,Q)$ of surjective homomorphisms from a free group $F_r$ to a finite group $Q$. 

\begin{lemma}\label{HomSizeLemma}
If $Q$ is a finite group that is minimally generated by $r_Q$ elements, then $\abs{\mathrm{Hom}_{\mathrm{sur}}(F_r,Q)} \geq \abs{Q}^{r-r_Q}$ for all $r \geq r_Q$.
\end{lemma}

\begin{proof}
Given $r \geq r_Q$, let $X_r = \set{x_1,\dots,x_r}$ and let $F_r = F(X_r)$ be the free group generated by $X_r$. We can view $F_{r_Q} \leq F_r$ by $F_{r_Q} = \innp{x_1,\dots,x_{r_Q}}$. Fixing $\varphi \in \mathrm{Hom}_{\textrm{sur}}(F_{r_Q},Q)$, for each $q_{r_Q+1},\dots,q_r \in Q$, we define $\Phi\colon F_r \to Q$ to be the unique homomorphism induced by the function $f\colon X_r \to Q$ given by
\[ f(x_j) = \begin{cases} \varphi(x_j), & j \leq r_Q, \\ q_j, & j > r_Q. \end{cases} \]
Since $\varphi$ is surjective, the homomorphisms $\Phi$ are surjective and distinct for all distinct (as ordered sets) choices of $q_{r_Q+1},\dots,q_r$. Hence $\abs{\mathrm{Hom}_{\textrm{sur}}(F_r,Q)} \geq \abs{Q}^{r-r_Q}$.
\end{proof}

We also require the following result of P.~Hall \cite{Hall}.

\begin{thm}\label{HallThm}
Let $Q$ be a non-abelian finite simple group and $\Gamma$ be a finitely generated group. If $\varphi_1,\dots,\varphi_m \in \mathrm{Hom}_{\mathrm{sur}}(\Gamma,Q)$ and $\varphi_i \ne \theta \circ \varphi_j$ for all $\theta \in \Aut(Q)$ and all $i \ne j$, then $\varphi_1 \times \dots \times \varphi_m\colon \Gamma \to Q^m$ is surjective.
\end{thm}

With all of the requisite material assembled, we now prove Proposition \ref{TechConProp}. 

\begin{proof}[Proof of Proposition \ref{TechConProp}]
We begin by setting $\mathcal{X}_r(Q) \bdef \mathrm{Hom}_{\textrm{sur}}(F_r,Q)/\Aut(Q)$ where the action of $\Aut(Q)$ on $\mathrm{Hom}_{\textrm{sur}}(F_r,Q)$ is given by post-composition. By Lemma \ref{HomSizeLemma}, we see that $\beta_{r,Q} = \abs{\mathcal{X}_r(Q)} \geq \alpha_Q^{-1}\abs{Q}^{r-r_Q}$ where $\alpha_Q = \abs{\Aut(Q)}$. For each equivalence class $x$ in $\mathcal{X}_r(Q)$, we fix a representative $\varphi_x \in \mathrm{Hom}_{\textrm{sur}}(F_r,Q)$. By Theorem \ref{HallThm}, we have a surjective homomorphism $\Phi_r\colon F_r \to Q^{\beta_{r,Q}}$ given by $\Phi_r =  \prod_{x \in \mathcal{X}_r(Q)} \varphi_x$. Fixing $Q = \PSL(2,\F_{29})$ and setting $P_1,P_2 \leq Q$ to be the $\Z$--equivalent subgroups given by Scott \cite{scott}, for each $m \in \N$ and $z = (z_i) = \set{1,2}^m$, we define $P_z \leq Q^m$ to be the subgroup $P_z \bdef \prod_{i=1}^m P_{z_i}$. It follows that for any distinct $z,z' \in \set{1,2}^m$ that $P_z,P_{z'}$ are $\Z$--equivalent and non-conjugate in $Q^m$. In particular, $Q^m$ has $2^m$ pairwise non-conjugate, pairwise $\Z$--equivalent subgroups. 

Now, given a large, non-arithmetic lattice $\Gamma \leq \mathrm{G}$ and $j \in \N$, we must find finite index subgroups $\Delta_1,\dots,\Delta_j \leq \Gamma$ that are pairwise non-isomorphic and pairwise $\Z$--equivalent. Since $\Gamma$ is non-arithmetic, combining Mostow--Prasad (see \cite{Mostow}, \cite{GPra}) and Margulis \cite[Thm 1, p. 2]{Margulis}, there exists a constant $C_\Gamma \in \N$ such that if $\Delta \leq \Gamma$ is a finite index subgroup, there are at most $C_\Gamma$ non-conjugate subgroups of $\Gamma$ that are isomorphic to $\Delta$ as an abstract group. Explicitly, $C_\Gamma = [\Comm(\Gamma):\Gamma]$ and so when $\Lambda \leq \Gamma$ is a finite index subgroup, we have $C_\Lambda = C_\Gamma [\Gamma:\Lambda]$. As $\Gamma$ is also large, there exists a finite index subgroup $\Gamma_2 \leq \Gamma$ and a surjective homomorphism $\psi\colon \Gamma_2 \to F_2$. Given any $r \geq 3$, there exists a subgroup $F_r \leq F_2$ of index $r-1$ such that $F_r$ is a free group of rank $r$. To see this, we first note that we have a surjective homomorphism $F_2 \to \Z$ given by sending $a=1$ and $b=0$, where $\set{a,b}$ is a free basis for $F_2$. We compose this surjection with the surjective homomorphism $\Z \to \Z/(r-1)\Z$ given by reduction modulo $r-1$. The kernel of the homomorphism $F_2 \to \Z \to \Z/(r-1)\Z$ has index $r-1$ in $F_2$. It follows by the Nielsen--Schreier theorem (see \cite[Thm 2.10]{MKS} for instance) that this subgroup of $F_2$ is free and of rank $r$. Setting $\Gamma_r = \psi^{-1}(F_r)$, we see that there exists subgroups $\Gamma_r \leq \Gamma_2 \leq \Gamma$ and surjective homomorphisms $\psi_r\colon \Gamma_r \to F_r$ with $[\Gamma_2:\Gamma_r] = r-1$. Now, for the given $j \in \N$,  we select $r$ such that $2^{\beta_{r,Q}} \geq j(r-1)C_{\Gamma_2}$. Note that this can be done since $\beta_{r,Q} \geq \alpha_Q^{-1}\abs{Q}^{r-2}$ grows exponentially as a function of $r$ whereas $(r-1)C_{\Gamma_2}$ only grows linearly as a function of $r$. By selection of $\Gamma_r$ and $r$, we have the surjective homomorphism
\[ \begin{tikzcd} \Gamma_r \arrow[twoheadrightarrow, rr,"\psi_r"', bend right = 45] \arrow[rrrr,twoheadrightarrow,"\mu_r", bend left = 45] & & F_r \arrow[rr, "\Phi_r"', twoheadrightarrow,bend right = 45]  & & Q^{\beta_{r,Q}}. \end{tikzcd} \]
For each $z \in \set{1,2}^{\beta_{r,Q}}$, we define $\Delta_z = \mu_r^{-1}(P_z)$ and note that the subgroups $\Delta_z$ are pairwise non-conjugate in $\Gamma_r$ and are pairwise $\Z$--equivalent. There are $2^{\beta_{r,Q}}$ such subgroups and we know that for each $\Delta_z$, there are at most $C_{\Gamma_r}$ subgroups from this list that can be abstractly isomorphic to a fixed $\Delta_z$. As $C_{\Gamma_r} = (r-1)C_{\Gamma_2}$ and $2^{\beta_{r,Q}} \geq j(r-1)C_{\Gamma_2}$, there is a subset of these subgroups of size at least $j$ that are all pairwise non-isomorphic.   
\end{proof}

\subsection{Completing the proof of Theorem \ref{Sec:Intro:T1}}

To prove Theorem \ref{Sec:Intro:T1} from Proposition \ref{TechConProp}, a few more words are required. As noted in the introduction, by work of Agol \cite[Thm 9.2]{Agol}, every closed hyperbolic $3$--manifold is large. In higher dimensions, using the construction of Gromov--Piatetski-Shapiro \cite{GPS}, there exists infinitely many commensurability classes of complete, finite volume hyperbolic $n$--manifolds that are both non-arithmetic and large. We note that there exist infinitely many commensurability classes of closed or complete, finite volume non-arithmetic hyperbolic $n$--manifolds for every $n$ follows directly from \cite{GPS}. That these examples are also large is well known. For the readers' sake, we briefly recall the construction of these manifolds with largeness in mind. First, we start with a pair of compact hyperbolic $n$--manifolds $M_1,M_2$ with connected, totally geodesic boundaries that are isometric. Gluing $M_1,M_2$ along the common boundaries $\partial M_1 = \partial M_2 = N$ produces a closed hyperbolic $n$--manifold $M$. By construction, $\pi_1(M) = \pi_1(M_1) \ast_{\pi_1(N)} \pi_1(M_2)$ and is large (see \cite[Thm 3.2]{Lub}). Lastly, using the construction of Deligne--Mostow \cite{DM}, there exist complete, finite volume complex hyperbolic 2--manifolds that are both non-arithmetic and large. As in the construction \cite{GPS}, Deligne--Mostow do not explicitly state that the non-arithmetic lattices they construct are large. That some of these lattices are large follows from the fact that they have surjective homomorphisms to hyperbolic triangle groups; see \cite[Thm 3.1]{Der}, \cite{Kap}, and \cite[Thm 3.1]{Toledo}. 

We can apply Proposition \ref{TechConProp} to any manifold $M$ in the above classes. We have opted to only write out the case when $M$ is a closed hyperbolic $n$--manifold as the complex hyperbolic setting is logically identical. Given $j \in \N$, $n \geq 3$, and a closed hyperbolic $n$--manifold $M$ which is non-arithmetic and large, we can apply Proposition \ref{TechConProp} with $\Gamma = \pi_1(M)$. We obtain $j$ pairwise non-isomorphic, finite index subgroup $\Delta_1,\dots,\Delta_j$ that are $\Z$--equivalent. By Corollary \ref{Sec:Coho:C1}, for any abelian group $A$ endowed with a trivial $\Z[\Gamma]$--module structure, we obtain compatible isomorphisms between the cohomology groups $H^k(\Delta_i,A)$ and $H^k(\Delta_{i'},A)$ for all $k$ and all $i,i'$. Since $M$ is aspherical, $M$ is a $K(\Gamma,1)$ for $\Gamma$. Setting $M_i$ to be the associated finite covers corresponding to $\Delta_i$, we see that $M_i$ is a $K(\Delta_i,1)$ for all $i$. In particular, we have that $H^k(M_i,A)$ and $H^k(\Delta_i,A)$ are compatibly isomorphic; the compatibility of the isomorphisms between $H^k(\Delta_i,A)$ and $H^k(\Delta_{i'},A)$ produce compatible isomorphisms between the cohomology groups $H^k(M_i,A)$ and $H^k(M_{i'},A)$. As the groups $\Delta_i,\Delta_{i'}$ are not isomorphic, by Mostow--Prasad rigidity (see \cite{Mostow}, \cite{GPra}) the manifolds $M_i,M_{i'}$ are not isometric. Taking $A = \Z$ produces (3) of Theorem \ref{Sec:Intro:T1}. The proof of Theorem \ref{Sec:Intro:T1} is completed by noting that $\Z$--equivalence implies $\Q$--equivalence and $\Q$--equivalence implies the manifolds $M_i,M_{i'}$ satisfy (1) and (2) by \cite{sunada}. 

\section{Proof of  parts (1) and (2) of Theorem \ref{thm:cohXi} }

Given a commutative ring $R$, a finite group $G$, and a finite $G$--set $X$, let $R[X]=\Hom_{sets}(X,R)$. This defines a contravariant functor from finite $G$--sets to $R[G]$--modules; if $p\colon X\to Y$ is $G$--map, let $p^*\colon R[Y]\to R[X]$ denote the corresponding homomorphism. When $p$ is onto, we have a homomorphism $p_*\colon R[X]\to R[Y]$ defined by $(p_*\phi)(y)=\sum_{x\in p^{-1}(y)} \phi(x)$. If $X=G$ and $Y=G/P$, then $(p_{*}\circ p^*)(\phi) = \abs{P}\phi$. It follows that $\Q[G/P]$ is a direct summand of $\Q[G]$, which can be identified with the left ideal $\Q[G]e_P$, where $e_P= \frac{1}{\abs{P}} \sum_{g\in P} g$ is the corresponding idempotent. It is convenient to normalize $p_*, p^*$ as follows. Instead of $p_*$ use the inclusion $\iota_P\colon \Q[G]e_P\to \Q$, and replace $p^*$ by the projection $p_P(x)= xe_P$. Given $\Q$--equivalent subgroups $P_1,P_2\leq G$ and set $e_i = e_{P_i}$, $\iota_i = \iota_{P_i}$, and $p_i = p_{P_i}$ for $i=1,2$. Since $\Q[G/P_i]$ are summands of $\Q[G]$ as $\Q[G]$--modules,  it follows that a $\Q[G]$--module isomorphism $f'\colon \Q[G/P_1]\to \Q[G/P_2]$ can be extended to $\Q[G]$--module isomorphism $f\colon \Q[G]\to \Q[G]$ such that the diagram
\[ \begin{tikzcd} \Q[G]\arrow[dd, "f"', bend right =15] \arrow[rr,"p_1", bend left = 30] & &  \Q[G/P_1]\arrow[dd, "f'", bend left =15] \arrow[ll, "\iota_1"', bend left = 30] \\ & & \\  \Q[G]\arrow[rr, "p_2"', bend left =30] & & \Q[G/P_2] \ar[ll, "\iota_2", bend left =30] \end{tikzcd} \]
commutes. By Skolem--Noether, the extension $f$ is necessarily right multiplication by an invertible element, that we will also denote by $f\in (\Q[G])^\times$. The commutativity implies that
\begin{equation}\label{eq:1}
e_2 = f^{-1} e_1 f
\end{equation}
We record this fact.

\begin{lemma}\label{lemma:1}
$(G,P_1,P_2)$ is a Gassman triple if and only if there exists $f\in (\Q[G])^\times$ such that \eqref{eq:1} holds.
\end{lemma}

The converse above is clear. If $f\in G$, then  \eqref{eq:1} says that $P_i$ are conjugate. Thus the Gassman condition is a weakening of conjugacy. Note that there are plenty of invertible elements of $\Q[G]$ which do not come from $G$. To see this, observe that by Artin--Wedderburn, $\overline{\Q}[G]$ is a product of matrix algebras. An element $f\in \Q[G]$ is invertible if and only the components of $f\otimes \overline{\Q}$ are invertible as matrices.

We now prove the first two parts of Theorem \ref{thm:cohXi}.  The remaining part will be proved in the next section. Recall that we are given $Y \in \Var_K$, where $\mathrm{char}(K) = 0$, $p\colon X \to Y$ is a Galois \'etale cover with Galois group $G$ and $X_i = X/P_i$, where $(G,P_1,P_2)$ is an $R$--Gassman triple. These fit into a diagram
\begin{equation}\label{eq:diamond}
\begin{tikzcd} 
& X\arrow[ld, "p_{X,1}"', bend right = 20, twoheadrightarrow,] \arrow[rd, twoheadrightarrow, "p_{X,2}", bend left = 20] \arrow[dd, "p", twoheadrightarrow] &  \\ X_1 \arrow[rd, "p_{1,Y}"', bend right = 20, twoheadrightarrow] &  & X_2\arrow[ld, "p_{2,Y}", bend left = 20, twoheadrightarrow] \\ & Y & 
\end{tikzcd}
\end{equation}
Our goal is to show that the cohomology groups $H^k(X_1,R)$ and $H^k(X_2,R)$ are isomorphic as Hodge structures or Galois modules.
We start with a rather simple proof of part (i) for rational coefficients.

\begin{proof}[First proof of Theorem \ref{thm:cohXi} (1) when $R=\Q$.]
In this case,  pullback $p_{X,i}^*$ gives an isomorphism of vector spaces $H^k(X_i(\C),\Q) \cong H^k(X(\C),\Q)^{P_i}$, with  inverse given by the normalized transfer $\frac{1}{|P_i|}(p_{X,i})_*$. Fixing a K\"ahler metric on $Y$, we endow the manifolds $X_1,X_2$ and $X$ with the pullback of this K\"ahler metric. The rational Hodge structures on these spaces are given by the standard Hodge-de Rham isomorphism between $H^k(X_i(\C),\Q)\otimes \C$ and the space of harmonic $k$--forms on $X_i$ in tandem with the decomposition of the latter into $(p,q)$--parts. As this data is compatible under pullback, we see that $H^k(X_i(\C),\Q) \cong H^k(X(\C),\Q)^{P_i}$ as Hodge structures. Applying Lemma \ref{lemma:1}, we deduce $H^k(X(\C),\Q)^{P_1}\cong H^k(X(\C),\Q)^{P_2}$ as Hodge structures.
\end{proof}

The above strategy will fail for integer coefficients,  because we cannot identify  $H^k(X_i,\Z)$ with $H^k(X,\Z)^{P_i}$ .
 So instead, we  push the coefficients down to $Y$.

\begin{proof}[Proof of Theorem \ref{thm:cohXi} (1) and (2)]
Suppose that $K=\C$. By covering space theory, $p\colon X\to Y$ corresponds to a surjective homomorphism $\rho\colon \pi_1(Y(\C))\to G$. Through $\rho$, any $R[G]$--module gives rise to  a local system of $R$--modules on $Y$. The local systems corresponding to the $R[G]$--module $R[G/P_i]$ are precisely the sheaves $(p_{i,Y})_*(R)$. It follows that $(p_{1,Y})_*(R)\cong (p_{2,Y})_*(R)$. Hence $H^k(Y(\C),(p_{1,Y})_*(R))\cong H^k(Y(\C),(p_{2,Y})_*(R))$. Since  the maps $p_{i,Y}$ are finite sheeted covers,  the Leray spectral sequences collapse to give isomorphisms
\begin{equation}\label{eq:HkXi}
 H^k(X_i(\C),R)\cong H^k(Y(\C),(p_{i,Y})_*(R))
\end{equation}
 Now suppose that $R=\Z$ or $\Q$. Using the language of  variations of Hodge structure  (see  \cite[\S 1-2]{zucker} for the relevant facts), the argument goes as follows. The local systems $(p_{i,Y})_*(R)$ can be regarded as variations of Hodge structures of type $(0,0)$ in a natural way. Consequently, the cohomology groups carry Hodge structures, and the isomorphisms \eqref{eq:HkXi} are compatible with these.  In  more explicit terms, if $V_i$ denotes the unitary flat bundle associated to $(p_{i,Y})_*(R)\otimes \C$, the Hodge structures result from the lattices $H^k(Y(\C),(p_{i,Y})_*(R))$ together with the isomorphisms of $H^k(Y(\C),(p_{i,Y})_*(R))\otimes \C$ to the  spaces of $V_i$--valued harmonic $k$--forms, plus the $(p,q)$ decompositions of the latter. This proves (1). 

The proof of (2) is formally identical, except that one works with the corresponding  \'etale notions \cite{deligne, milne}. Let us assume that $R=\Z_p$ as
 the argument for $\Q_p$ is the same. \'Etale covers of $Y$ are classified by open subgroups of the \'etale fundamental group $\pi_1^{et}(Y)$, which is an extension of $\Gal(\overline{K}/K)$ by the profinite completion of $\pi_1(Y(\C))$; this depends on the choice of a base point.  In particular, $X$ corresponds to a surjective continuous homomorphism  $\rho\colon \pi^{et}_1(Y)\to G$. The local systems (more precisely lisse sheaves, see \cite[Rapport]{deligne}) $(p_{i,Y})_*(\Z_p)$ correspond to the representations of the \'etale fundamental groups $\pi^{et}_1(Y)\to \Z_p[G/P_i]$ defined as above, and these are isomorphic. The cohomology of these sheaves come with canonical Galois actions, and we have isomorphisms $H_{et}^k(X_{i,\overline{K}},\Z_p)\cong H_{et}^k(Y_{\overline{K}},(p_{i,Y})_*(\Z_p))$ compatible with Galois actions. (This  is discussed in  \cite[Rapport, \S1.2-1.4]{deligne} when $K$ is a  finite field, but the same reasoning applied here.) This proves (2).
\end{proof}

\begin{rem}
If the varieties in \eqref{eq:diamond} are replaced by $\Z$--equivalent manifolds $X_1$ and $X_2$, the same argument as above shows that
$H^k(X_1,\Z)\cong H^k(X_2,\Z)$.  This answers Question 2 in Prasad \cite{prasad}. For aspherical manifolds, this also follows from
Corollary \ref{Sec:Coho:C1}.
\end{rem}

\section{Motives}\label{Sec:motives}

An additive category  $C$ is called \textbf{pseudo-abelian} if every idempotent (i.e.~$p^2=p$) morphism $p\colon V\to V$ has a kernel and $V\cong \ker (p)\oplus \ker(1-p)$. The image of  an idempotent $p$ also exists, and is given by $p(V)= \ker(1-p)$. Fixing a pseudo-abelian $\Q$--linear category $C$ and object $V$ on which a finite  group $G$ acts by automorphisms, we have a homomorphism $\Q[G]\to \End_C(V)$ of algebras. Given a subgroup $P\leq G$, we define $V^P\subset V$ to be the image of the idempotent $e_P= \frac{1}{\abs{P}} \sum_{g\in P} g$. 

\begin{lemma}\label{lemma:VPi}
If $V$ is as above with $P_1, P_2 \leq G$ are $\Q$--equivalent subgroups, then $V^{P_1}\cong V^{P_2}$. 
\end{lemma}

\begin{proof}
Let $f\in \Q[G]$ be as in Lemma \ref{lemma:1}, then $f\colon V^{P_2} \to V^{P_1}$ is an isomorphism.
\end{proof}

Let $\Var_K$ denote the category of smooth projective varieties over a field $K$ and $\CH^*(X)$ denote the Chow ring of cycles modulo rational equivalence tensored with $\Q$ (see \cite{fulton} for instance). We can form the category $\Cor_K$ of (degree $0$) correspondences: the objects are the same as $\Var_K$, $\Cor(X,Y) = \CH^d(X\times Y)$, where $d=\dim X$ (more details can be found in \cite{fulton, kleiman, scholl}.) The category of effective Chow motives $\Mot^{eff}_K$ is the pseudo-abelian completion of the previous category. More concretely, an object of $\Mot^{eff}_K$ is given by a pair $(X,e)$,
where $e\in \Cor(X,X)$ is an idempotent. Morphisms are given by
\[ \Hom_{\Mot^{eff}}((X,e), (X',e'))= \frac{\{ f\in \Cor(X,X')\mid f\circ e= e'\circ f\}}{\{f\mid f\circ e=e'\circ f=0\}}. \]
Set $M(X) = (X,id)$, which is the \textbf{motive associated to $X$}, and $(X,e)= e(M(X))$. Suppose that a finite group $G$ acts on $X\in \Var_K$. Then we can embed $\Q[G]\subset \Cor(X,X)$, by sending $g$ to the graph of the corresponding automorphism of $X$.

\begin{lemma}
Suppose that $Y\in \Var_K$ and $p\colon X\to Y$ is an Galois \'etale cover with Galois group $G$. Then $[Y]\cong (X,e_G)$, where $e_G=\frac{1}{|G|}\sum_{g\in G} g$.
\end{lemma}

\begin{proof}
The graph of $p$ defines an element of $\Hom_{\Mot^{eff}}((X,e), Y)$ that we must show is an isomorphism. By Manin's identity principle \cite{scholl}, it is enough to check that $\CH^*((X,e)\otimes Z)\to \CH^*(Y\otimes Z)$ is an isomorphism for every $Z\in \Var_K$.  This map is $\CH^*(X\otimes Z)^G\to \CH^*(Y\otimes Z)$ which is an isomorphism by \cite[1.7.6]{fulton}.
\end{proof}

\begin{cor}\label{Sec:Motive:C1}
$M(Y)\cong e_G(M(X))=M(X)^G$.
\end{cor}

The next result will complete the proof  of Theorem \ref{thm:cohXi}. Recall, we are given a $\Q$--Gassman triple $(G,P_1,P_2)$, a $G$--\'etale cover $X\to Y$ with $X_i = X/P_i$ and $Y \in \Var_K$ where $\mathrm{char}(K) = 0$. 

\begin{prop}\label{prop:motivesX1X2}
$M(X_1)\cong M(X_2)$ in $\Mot^{eff}_K$.
\end{prop}

\begin{proof}
By Lemma \ref{lemma:VPi} and Corollary \ref{Sec:Motive:C1}, we have $M(X_1) \cong M(X)^{P_1}\cong M(X)^{P_2}\cong M(X_2)$.
\end{proof}

The category of motives $\Mot_K$ is obtained by inverting the so called Lefschetz object in $\Mot_K^{eff}$, c.f. \cite{kleiman} (in \cite{andre, scholl}, $\Mot_K$ is constructed from $\Cor_K$ in one step).

\begin{cor}
The motives of $X_1$ and $X_2$ in $\Mot_K$ are isomorphic.
\end{cor}

\begin{rmk}
 Since $H^*(X(\C),\Q)$ and $H_{et}^*(X,\Q_p)$ depend on the underlying motives, we recover Theorem \ref{thm:cohXi} (1) and (2)  for these  coefficients.
Although the previous arguments were more direct.
\end{rmk}

We now prove Theorem \ref{prop:ZequivVar}. Recall that this  says that there are arbitrarily large collections of projective surfaces over $\overline{\Q}$ with distinct fundamental groups (with respect to a fixed embedding $\overline{\Q}\subset \C$) but isomorphic motives. 

\begin{proof}[Proof of Theorem \ref{prop:ZequivVar}]
From  Proposition \ref{TechConProp}, we deduce that there are $j$ pairwise non-isomorphic $\Z$--equivalent compact torsion free lattices $\Delta_i\leq \SU(2,1)$. These act on $\Hy_\C^2$ which can be identified with the complex $2$--ball $B\subset \C^2$. Setting $X_i = B/\Delta_i$, we note that these spaces are projective algebraic by Kodaira's embedding theorem \cite[pp 219--220]{wells}. Each $X_i$ is also rigid Calabi--Vesentini  \cite{CalabiVesentini} and hence defined over $\overline{\Q}$. By construction $\pi_1(X_i)=\Delta_i\ncong \Delta_{i'}=\pi_1(X_{i'})$ when $i\not= i'$. The remaining properties follow from Theorem \ref{thm:cohXi}.
\end{proof}

Let $K_0(\Var_K)$ denote the Grothendieck ring of $K$--varieties.  When $\mathrm{char}(K)=0$, a nice presentation was given by Bittner \cite{bittner}: The generators are isomorphism classes $[X]$ of smooth projective varieties, and $[\Bl_ZX] - [E]= [X]-[Z]$ holds whenever $\Bl_ZX$ is the blow up of $X$ along a smooth subvariety $Z\subset X$ with exceptional divisor $E$. Using this presentation together with the formulas in \cite[pp 77--78]{kleiman}, we get a surjective ring homomorphism 
\[ \chi_m^{eff}\colon K_0(\Var_{\overline{\Q}})\longrightarrow K_0(\Mot_{\overline{\Q}}^{eff}) \]
sending $[X]\mapsto [M(X)]$. This can be thought of as the motivic Euler characteristic. It is natural to ask whether this is an isomorphism; in some form this question goes back to Grothendieck \cite[p 174]{gs}. The right side is a $\Q$--algebra because $\Mot_{\overline{\Q}}$ is $\Q$--linear. Therefore we have and induced homomorphism
\[ \chi_m^{eff}\otimes \Q\colon K_0(\Var_{\overline{\Q}})\otimes \Q\longrightarrow K_0(\Mot_{\overline{\Q}}^{eff}). \]

\begin{thm}\label{Sec:Motive:NonInject}
The homomorphism $\chi_m^{eff}\otimes \Q$ is not injective.
\end{thm}

Before proving Theorem \ref{Sec:Motive:NonInject}, we require the following lemma.
Let $\Grp$ be the set of isomorphism classes of  finitely generated  groups. This becomes a commutative
monoid under  the operation $[G_1][G_2]= [G_1\times G_2]$.
Let $\Q[\Grp]$ denote the monoid algebra associated to $\Grp$.

\begin{lemma}\label{Sec:Motive:L-end}
There is a ring homomorphism $K_0(\Var_{\overline{\Q}})\otimes \Q\to \Q[\Grp]$ which sends $[X]\to [\pi_1(X(\C))]$.
\end{lemma}

\begin{proof}
Two varieties $X_1,X_2$ are stably birational if $X_1\times \PP^n$ is birational to $X_2\times \PP^m$ for some $n,m$. Let $\SB_{\overline{\Q}}$ denote the set of stable birational classes of smooth projective varieties defined over $\overline{\Q}$. Products of varieties makes this into a commutative monoid.  By a theorem of Larsen--Lunts \cite{larslunts}, there exists a homomorphism $\lambda\colon K_0(\Var_{\overline{\Q}})\otimes \Q\to \Q[\SB_{\overline{\Q}}]$ sending $[X]$ to the stable birational class of $X$. Since stably birationally equivalent smooth projective varieties have isomorphic fundamental groups, $X\mapsto \pi_1(X(\C))$ induces homomorphism of monoids $\SB_{\overline{\Q}}\to \Grp$ and of rings $\Q[\SB_{\overline{\Q}}]\to \Q[\Grp]$. 
 Compose this with $\lambda$ to get  the desired homomorphism.
\end{proof}

\begin{proof}[Proof of Theorem \ref{Sec:Motive:NonInject}]
Taking $X_i$ as in Theorem  \ref{prop:ZequivVar}, we see that $[X_1]-[X_2]$ lies in the kernel and is non-zero by Lemma \ref{Sec:Motive:L-end}.
\end{proof}

\begin{cor}
 The composition $\chi_m\otimes \Q\colon K_0(\Var_{\overline{\Q}})\otimes \Q\to K_0(\Mot_{\overline{\Q}})$ is also not injective.
\end{cor}

This statement can also be deduced from work of Borisov \cite{borisov}, who
shows that the Lefschetz class $\LL = [\PP^1]-[\mathrm{pt}]\in K_0(\Var_{\overline{\Q}})$ is a zero divisor. Elements annihilated by $\LL$ must lie
in the kernel  of $\chi_m$ because $\chi_m(\LL)$ is invertible.



\end{document}